\documentclass[a4paper,twopage,reqno,12pt]{amsart}
\usepackage[top=30mm,right=30mm,bottom=30mm,left=30mm]{geometry}

\usepackage{graphicx}
\usepackage{amsmath}
\usepackage{amssymb}
\usepackage{amsfonts}
\usepackage{amsthm}
\usepackage{amstext}
\usepackage{amsbsy}
\usepackage{amsopn}
\usepackage{amscd}
\usepackage{enumerate}
\usepackage{color}
\usepackage[colorlinks]{hyperref}
\usepackage[hyperpageref]{backref}
\usepackage{multirow}
\usepackage{float}
\usepackage{color}
\usepackage[colorlinks]{hyperref}
\hypersetup{
colorlinks=true,%
citecolor=green,%
filecolor=black,%
linkcolor=blue,%
urlcolor=magenta
}

\usepackage{longtable}
\usepackage{lscape}

\newtheorem{theorem}{Theorem}[section]
\newtheorem{corollary}[theorem]{Corollary}
\newtheorem{lemma}[theorem]{Lemma}
\newtheorem{proposition}[theorem]{Proposition}

\theoremstyle{definition}

\newtheorem{remark}[theorem]{Remark}

\numberwithin{equation}{section}

\newcommand{\m}{\textsf{m}}
\newcommand{\n}{\textsf{n}}
\newcommand{\nse}{\tau}
\newcommand{\oc}{\textsf{oc}}
\newcommand{\Aut}{\textsf{Aut}}
\newcommand{\Out}{\textsf{Out}}

\newcommand{\cl}{\textsf{cl}}
\newcommand{\Syl}{\textsf{Syl}}

\newcommand{\A}{\mathrm{A}}

\newcommand{\F}{\mathrm{F}}

\newcommand{\G}{\mathrm{G}}

\newcommand{\PSL}{\mathrm{PSL}}
\newcommand{\PSU}{\mathrm{PSU}}

\newcommand{\PSp}{\mathrm{PSp}}

\newcommand{\POm}{\mathrm{P \Omega}}

\newcommand{\E}{\mathrm{E}}

\renewcommand{\S}{\mathrm{S}}

\newcommand{\D}{\mathrm{D}}

\newcommand{\f}{\mathbf{f}}

\newcommand{\B}{\mathrm{B}}

\newcommand{\Zbb}{\mathbb{Z}}

\newcommand{\Fbb}{\mathbb{F}}

\newcommand{\Bmc}{\mathcal{B}}

\newcommand{\Amc}{\mathcal{A}}
\newcommand{\e}{\epsilon}

\renewcommand{\leq}{\leqslant}
\renewcommand{\geq}{\geqslant}

\renewcommand{\bar}{\overline}

\newcommand{\mmod}[3]{#1 \overset{#3}{\equiv} #2}

\begin{document}
\title[]{On quantitative structure of symplectic groups}

\author[S. H. Alavi]{Seyed Hassan Alavi$^*$}
\address{S. H. Alavi, Department of Mathematics, Faculty of Science, Bu-Ali Sina University, Hamedan, Iran}
\email{alavi.s.hassan@basu.ac.ir}
\thanks{Corresponding author: S.H. Alavi}

\author[A. Daneshkhah]{Ashraf Daneshkhah}

\address{A. Daneshkhah, Department of Mathematics, Faculty of Science, Bu-Ali Sina University, Hamedan, Iran}
\email{adanesh@basu.ac.ir}

\author[H. Parvizi Mosaed]{Hosein Parvizi Mosaed}
\address{H. Parvizi Mosaed, Alvand Institute of Higher Education, Hamedan, Iran}
\email{h.parvizi.mosaed@alvand.ac.ir}

\maketitle

\begin{abstract}
 The main aim of this article is to study quantitative structure of projective symplectic groups $\PSp_{4}(q)$ with $q>2$ even. Indeed, we prove that the groups $\PSp_{4}(q)$ with $q>2$ even are uniquely determined by their orders and the set of the number of elements of the same order.  This result links to the well-known J. G. Thompson's problem (1987) for finite simple groups.
\end{abstract}

\section{Introduction}\label{sec:Intro}

For a finite group $G$ and a positive integer $n$, let $G_{n}$ consist of all elements $x$ satisfying $x^{n} = 1$. The \emph{type} of $G$ is defined to be the function whose value at $n$ is the order of $G_{n}$. This notion links to the notion of  \emph{same-order type} $\nse(G)$ of $G$, that is to say, \emph{the set of the number of elements of the same order} in $G$. Indeed, it turns out that if two groups $G$ and $H$ are of the same type, then $\nse(G)=\nse(H)$ and $|G|=|H|$. However, finite groups with the same order and same-order type may not have the same type or even the same \emph{spectrum}, the set of element orders. 
The smallest examples of such groups that we have found are $G:=\Zbb_{4}\times (\Zbb_{7} : \Zbb_{3})$ and $H:=\Zbb_{3}\times (\Zbb_{7} : \Zbb_{4})$ of order $84$ in which $\nse(G)=\nse(H)=\{1, 2, 6, 12, 14, 28\}$, and $G$ has an element of order $28$ while $H$ has no such an element. Moreover, $G$ and $H$ are not of the same type as $|G_{3}|=15$ and $|H_{3}|=3$. 

The main aim of this paper is to investigate the structure of groups with the same-order type as finite simple groups. In this direction, Shao et al \cite{art:Shao} studied finite simple groups whose order is divisible by at most four primes. This problem has been studied for several families of simple groups including small Ree groups and Suzuki groups \cite{art:Alavi2017,art:Alavi2019}. In this paper, we study groups with the same order and the same-order type as $\PSp_{4}(q)$ with $q$ even:   

\begin{theorem}\label{thm:main}
	Let $G$ be a group such that $\nse(G)=\nse(\S)$ and $|G|=|\S|$, where $\S=\PSp_{4}(q)$ with $q>2$ even. Then $G$ is isomorphic to $\S$.
\end{theorem}

In 1987, J. G. Thompson \cite[Problem 12.37]{book:khukh} asked \emph{if a group with the same type as a solvable group is still solvable}. One may ask this problem for non-solvable groups, in particular, finite simple groups. It is easy to see that two groups with the same type share the same order and spectrum, and hence the recognition of finite simple groups by their orders and spectra, which is known as Shi's problem \cite[Problem 12.39]{book:khukh}, gives a positive answer to Thompson's problem. It is worth noting that  Vasil$'$ev, Grechkoseeva and Mazurov \cite{a:spec} finalized a complete solution to Shi's problem, and consequently Thompson's problem is true for finite simple groups. As noted above, two groups $G$ and $H$ with the same type have the same order and  same-order type, and hence Theorem \ref{thm:main} also gives  a positive answer to Thompson's problem for finite simple groups $\PSp_{4}(q)$ with $q$ even:

\begin{corollary}\label{cor:main}
	If $G$ is a finite group of the same type as $\S:=\PSp_{4}(q)$ with $q>2$ even, then $G$ is isomorphic to $\S$.
\end{corollary}

Our analysis in proving Theorem~\ref{thm:main} depends on determining $\nse(\S)$, where $\S=\PSp_{4}(q)$ with $q>2$ even, see Proposition~\ref{prop:nse(S)}. Despite of the investigation of the groups which have been previously considered, we have generally no isolated vertex in the prime graph of a group $G$ with the same number of elements of the same order as $\S$. 

\subsection{Definitions and notation}\label{sec:defn}

In this section, we give some brief comments on the notation used in this paper. Throughout this article all groups are finite. Our group theoretic notation is standard, and it is consistent with the notation in \cite{book:Car,book:atlas,book:Gor}. In particular, we follow the  notation below for the finite simple classical groups: 
\begin{enumerate}[]
	\item[] $\PSL_{n}(q)$, for $n\geq 2$ and $(n,q)\neq (2,2), (2,3)$,
	\item[] $\PSU_{n}(q)$, for $n\geq 3$ and $(n,q)\neq (3,2)$,
	\item[] $\PSp_{2n}(q)$, for $m\geq 2$ and $(n,q)\neq (2,2)$,
	\item[] $\POm_{2n+1}^{\circ}(q)$, for $m\geq 3$ and $q$ odd,
	\item[] $\POm_{2n}^{\pm}(q)$, for $m\geq 4$.
\end{enumerate}
In this manner, the only repetitions are:
\begin{enumerate}[]\label{eq:iso}
	\item[] $\PSL_{2}(4)\cong \PSL_{2}(5)\cong \A_{5}$, $\PSL_{2}(7)\cong \PSL_{3}(2)$, $\PSL_{2}(9)\cong \A_{6}$,
	\item[] $\PSL_{4}(2)\cong \A_{8}$ and $\PSp_{4}(3)\cong \PSU_{4}(2)$.
\end{enumerate} 
We denote the  set of Sylow $p$-subgroups of $G$ by $\Syl_{p}(G)$. We also use $\n_p(G)$ to  denote the number of Sylow $p$-subgroups of $G$. For a positive integer $n$, the set of prime divisors of $n$ is denoted by $\pi(n)$, and if $G$ is a finite group, $\pi(G):=\pi(|G|)$, where $|G|$ is the order of $G$. We denote the set of elements' orders of $G$ by $\omega(G)$ known as \emph{spectrum} of $G$. The \emph{prime graph} $\Gamma(G)$ of a finite group $G$ is a graph whose vertex set is $\pi(G)$, and two distinct vertices $p$ and $q$ are adjacent if and only if $pq\in\omega(G)$. Assume further that $\Gamma(G)$ has $t(G)$ connected components $\pi_i:=\pi_i(G)$, for $i=1,2,\hdots,t(G)$. The positive integers $n_{i}$ with $\pi(n_{i})=\pi_{i}$ are called \emph{order components} of $G$. Clearly, $|G|=n_1\cdots n_{t(G)}$. We denote by $\oc(G)$ the set of all order components of $G$. In the case where $G$ is of even order, we always assume that $2\in\pi_1$, and $\pi_{1}$ is said to be the even component of $G$. In this way, $\pi_{i}$ and $n_{i}$ are called odd components and odd order components of $G$, respectively. Recall that $\nse(G)$ is the set of the number of elements in $G$ with the same order. In other word, $\nse(G)$ consists of the numbers $\m_i(G)$ of elements of order $i$ in $G$, for $i\in \omega(G)$. Here, $\varphi$ and $\psi$ are the \emph{Euler totient} function and the  \emph{Dedekind arithmetic} function, respectively. In the other word, for every positive integer $n$,
\begin{equation*}
	\varphi(n)= n\prod_{p\in \pi(n)}(1 - \frac{1}{p}) \text{\quad  and \quad }\psi(n)= n\prod_{p\in \pi(n)}(1 + \frac{1}{p}).
\end{equation*} 
We also use $\Phi_{n}(q)$ to denote the \emph{$n$-th cyclotomic polynomial} of order $n$.

\section{Preliminaries}\label{sec:preli}
In this section, we give some useful results which we frequently use in order to prove our results.

\begin{lemma}{\rm \cite[Theorem 1]{art:Chen}}\label{lem:frob}
	Let $G$ be a Frobenius group of even order with kernel $K$ and complement $H$. Then  $\pi(H)$ and $\pi(K)$ are the only connected components of $\Gamma(G)$.
\end{lemma}

A group $G$ is called 2-Frobenius group if there exists a normal series $1\unlhd H\unlhd K\unlhd G$ such that $G/H$ and $K$ are Frobenius groups with kernels $K/H$ and $H$, respectively.

\begin{lemma}{\rm \cite[Theorem 2]{art:Chen}}\label{lem:2frob}
	Let $G$ be a $2$-Frobenius group of even order. Then $\Gamma(G)$ has exactly two connected components. Moreover, there exists a normal series $1\unlhd H\unlhd K\unlhd G$ such that
	\begin{enumerate}[{\quad \rm (i)}]
		\item $\pi_1(G)=\pi(H)\cup\pi(G/K)$, and $\pi_2(G)=\pi(K/H)$;
		\item $G/K$ and $K/H$ are cyclic groups;
		\item $|G/K|$ divides $|\Aut(K/H)|$, and $(|G/K|,|K/H|) = 1$;
		\item $H$ is a nilpotent group and $G$ is a solvable group.
	\end{enumerate}
\end{lemma}

\begin{lemma}{\rm \cite{art:kondratev,art:Williams}}\label{lem:graph}
	Let $G$ be a finite group with more than one prime graph component. Then one of the following statements holds:
	\begin{enumerate}[{\quad \rm (i)}]
		\item $G$ is a Frobenius group;
		\item $G$ is a $2$-Frobenius group;
		\item $G$ has a normal series $1\unlhd H\unlhd K\unlhd G$ such that $H$ and $G/K$ are $\pi_1(G)$-groups, $K/H$ is a non-abelian simple group, $H$ is a nilpotent group, $|G/K|$ divides $|\Out(K/H)|$, $t(K/H)\geq t(G)$, and for any $i\in \{2,\ldots,t(G)\}$, there exists $j\in\{2,\ldots, t(K/H)\}$ such that $\pi_i(G)=\pi_j(K/H)$.
	\end{enumerate}
\end{lemma}

\begin{lemma}{\rm \cite[p. 50]{book:rose}}\label{lem:cyclic}
	All subgroups and all quotient groups of any cyclic group are cyclic.
	If $G$ is a cyclic group of finite order $n$, then $G$ has just one subgroup $H$ of order $m$ for each divisor $m$ of $n$, $H$ is cyclic and $G/H$ is cyclic of order $n/m$.
\end{lemma}

\begin{lemma}{\rm \cite[ p. 35]{book:rose}}\label{lem:composable}
	Let $n$ be an integer, $n>1$, and let the factorization of $n$ as a product of primes be $n=p_1^{m_1}p_2^{m_2}\ldots p_r^{m_r}$,
	where $r, m_1,\ldots, m_r$ are positive integers and $p_1,\ldots, p_r$ distinct primes. Then $\Zbb_n$ is decomposable if $r>1$, and
	$\Zbb_n\cong \Zbb_{q_1}\times \Zbb_{q_2}\times \ldots\times \Zbb_{q_r}$,
	where $q_i= p_i^{m_i}$ for each $i=1,\ldots, r$. Moreover, for each prime $p$ and positive integer $m$, $ _{p^m}$ is indecomposable.
\end{lemma}

\begin{lemma}{\rm \cite[ p. 198]{book:rose}}\label{lem:G^n}
	Let $n$ be a positive integer, and let $G$ be an abelian group. Also let $G_n=\{x\in G ~ : ~ x^n=1\}$ and $G^n=\{x^n ~ : ~ x\in G\}$. Then the following statements hold:
	\begin{enumerate}[{\rm \quad (i)}]
		\item $G_n\leqslant G$, and $G/G_n \cong G^n$;
		\item Suppose that $G$ is finite, so that there are elements $x_1,\ldots,x_r$ of $G$ such that $G=\langle x_{1} \rangle \times \cdots \times \langle x_r \rangle$.
		For each $i=1,\ldots,r$, let $n_i=o(x_i)$ and $k_i=n_i/(n_i,n)$. Then $G_n=\langle x_{1}^{k_{1}} \rangle \times \cdots \times \langle x_r^{k_{r}} \rangle$;
		\item If $G$ is finite, then $G/G^n\cong G_n$.
	\end{enumerate}
\end{lemma}

\begin{lemma}{\rm \cite[Theorem 9.1.2]{book:Hall}}\label{lem:G_n}
	Let $G$ be a finite group, and let $n$ be a positive integer dividing $|G|$. Then $n$ divides $|G_n|$.
\end{lemma}

The proof of the following result is straightforward by Lemma \ref{lem:G_n}.

\begin{lemma}\label{lem:euler}
	Let $G$ be a finite group. Then for every $i\in\omega(G)$, $\m_i(G)=k\varphi(i)$, where $k$ is the number of cyclic subgroups of order $i$ in $G$, and $i$ divides $\sum_{j \mid i} \m_j(G)$. Moreover, if $i>2$, then $\m_i(G)$ is even.
\end{lemma}

\begin{lemma}{\rm \cite[Theorem 3]{art:Weisner}}\label{lem:multi}
	Let $G$ be a finite group. Then the number of elements whose orders are multiples of $n$ is either zero, or a multiple of the greatest divisor of $|G|$ that is prime to $n$.
\end{lemma}

\begin{lemma}{\rm \cite[Lemma 2.6]{art:zhang}}\label{lem:oc}
	Let $G$ be a finite group, and $\PSp_4(q)$ be a projective symplectic group, where $q > 3$. If $\oc(G)=\oc(\PSp_4(q))$, then $G \cong \PSp_4(q)$.
\end{lemma}

\begin{lemma}{\rm \cite{art:Crescenzo}}\label{lem:numb}
	Let $p$ and $q$ be primes such that $p^{m}=q^{n}+1$, for some positive integers $m$ and $n$. Then one of the following statements holds.
	\begin{enumerate}[\rm (i)]
		\item $(p,q,m,n)=(3,2,2,3)$;
		\item $p=2^{n}+1$ is a Fermat prime with $n$ a power of $2$;
		\item $q=2^{m}-1$ is a Mersenne prime with $m$ prime.
	\end{enumerate}
\end{lemma}

\begin{lemma}\label{lem:q1}
	Let $q$ be a power of $2$. Then 
	\begin{enumerate}[\rm (i)]
		\item $2q^2+3$ does not divide $q^4(q^4-1)(q^2-1)$;
		\item If $q^2+2$ divides $q^4(q^4-1)(q^2-1)$, then $q=2,4$;
		\item If $2q^2+1$ divides $q^4(q^4-1)(q^2-1)$, then $q=2$;
		\item If $3q^2+2$ divides $q^4(q^4-1)(q^2-1)$, then $q=4$;
		\item $q^4-9$ does not divide $q^4(q^4-1)(q^2-1)$.
	\end{enumerate}
\end{lemma}
\begin{proof}
	We know that $2q^2+3$ is coprime to $q^4(q^2+1)$. If $2q^2+3$  would divide $q^4(q^4-1)(q^2-1)$, then $2q^2+3$ would divide $(q^2-1)^{2}$. Since $4(q^2-1)^2=(2q^2-7)(2q^{2}+3)+25$, it would follow that $2q^2+3\mid 25$, and so $q^{2}\leq 11$ which is true when $q=2$ but in this case $2q^{2}+3=11$ does not divide $25$. This follows part (i). To prove part (ii), since $q^{2}+2$ is coprime to $q^{2}+1$, if $q^2+2$ divides $q^4(q^4-1)(q^2-1)$, we conclude that $q^{2}+2$ has to divide $q^{4}(q^{2}-1)^{2}$. We know that $q^{4}(q^{2}-1)^{2}=(q^6-4q^4+9q^2-18)(q^2+2)+36$. Then $q^{2}+2$ divides $36$, and this is true if and only if $q=2,4$, as desired. Parts (iii)-(v) can be proved in a same manner. 
\end{proof}

\section{Properties of projective symplectic groups}\label{sec:psp}
In this section, we state some useful facts about projective symplectic groups of dimension $4$ and their element orders in order to obtain the number of  the elements of the same order of $\PSp_4(q)$ with $q=2^{f}>2$  in Section \ref{sec:nse} below.

\subsection{Conjugacy classes}

Let $\bar{\Fbb}$ be the algebraic closure of the finite field $\Fbb$ of size $q=p^f$. Let also $ \Fbb_i=\{x\in \bar{\Fbb}~ : ~ x^{q^i}=x\}$ be the subfield of $\bar{\Fbb}$ with $q^i$ elements, and let $\Fbb_1:=\Fbb$.
Let now $\kappa$ be a fixed generator of the multiplicative group $\Fbb_4^\times$, and set
\begin{align*}
	\text{$\tau:=\kappa^{q^2-1}$, $\theta:=\kappa^{q^2+1}$, $\eta:=\theta^{q-1}$ and $\gamma:=\theta^{q+1}$.}
\end{align*}
Then $\Fbb_2^\times=\langle\theta\rangle$  and $\Fbb^\times=\langle\gamma\rangle$. 
For $t\in \Fbb$ and $z_i\in \Fbb_4^\times$, define
\begin{align}\label{eq:x}
	\nonumber  x_a(t)&=
	\begin{bmatrix}
		1 & t &   &   \\
		& 1 &   &    \\
		&   & 1 & -t \\
		&   &   & 1  \\
	\end{bmatrix},
	&
	x_{a+b}(t)&=
	\begin{bmatrix}
		1 &   & t &    \\
		& 1 &   & t  \\
		&   & 1 &    \\
		&   &   & 1  \\
	\end{bmatrix},
	\\
	x_b(t)&=
	\begin{bmatrix}
		1 &   &   &   \\
		& 1 & t &   \\
		&   & 1 &   \\
		&   &   & 1 \\
	\end{bmatrix},
	&
	x_{2a+b}(t)&=
	\begin{bmatrix}
		1 &   &   & t \\
		& 1 &   &   \\
		&   & 1 &   \\
		&   &   & 1 \\
	\end{bmatrix},
	\\
	\nonumber    h(z_1,z_2)&=
	\begin{bmatrix}
		z_1 &     &          &           \\
		& z_2 &          &           \\
		&     & z_2^{-1} &           \\
		&     &          & z_1^{-1}  \\
	\end{bmatrix}.
\end{align}
Let $\omega_\iota:=x_{\iota}(1)' x_{\iota}(-1) x_{\iota}(1)$ for $\iota\in\Delta^+ :=\{a, b, a+b, 2a+b\}$, where $x_{\iota}(1)'$ is the transposed matrix of $x_{\iota}(1)$. In particular,

\begin{align}\label{eq:om}
	\omega_a=
	\begin{bmatrix}
		& 1 &   &    \\
		-1 &   &   &    \\
		&   &   & -1 \\
		&   & 1  &    \\
	\end{bmatrix}
	\text{ and }
	\omega_b=
	\begin{bmatrix}
		1 &    &   &   \\
		&    & 1 &   \\
		& -1 &   &   \\
		&    &   & 1 \\
	\end{bmatrix}.
\end{align}

Therefore, the $4$-dimensional projective symplectic group $\PSp_4(q)$ over $\Fbb$ is generated by $x_{\iota}(t)$, $h(z_1,z_2)$, $\omega_{a}$ and $\omega_{b}$ defined as in \eqref{eq:x} and \eqref{eq:om}. We use the following set of parameters for the conjugacy classes: for $i,j\in \Zbb$, ``$\Zbb \mod m$'' means that $i$ and $j$ give the same class if $i\equiv j$ $\pmod m$.
\begin{align*}
	T_0=& \{i\in \Zbb \pmod {q-1}\},\\
	T_1=& \{i\in T_0 : i\not\equiv 0 \pmod{q-1}\},\\
	T_2=&\{i\in \Zbb  \pmod{q+1} : i\not\equiv 0\pmod{q+1}\},\\
	S_0=& \{(i,j)\in T_0\times T_0 : i\not\equiv j\pmod{q-1}\},\\
	S_1=& \{(i,j)\in T_1\times T_1 : i\not\equiv \pm j\pmod{q-1}\},\\
	S_2=& \{(i,j)\in T_2\times T_2 : i\not\equiv \pm j\pmod{q+1}\},\\
	R_1=& \{i\in \Zbb\pmod{q^2-1} : i \not\equiv qi\pmod{q^2-1}\},\\
	R_2=& \{i\in \Zbb\pmod{q^2-1} : i \not\equiv \pm qi\pmod{q^2-1}\},\ \text{and} \\
	R_3=& \{i\in \Zbb  \pmod{q^2+1} : i\not\equiv 0\pmod{ q^2+1}\}.
\end{align*}

\begin{remark}\label{rem:conj}
	In \cite{art:Enomoto}, the detail information of the conjugacy classes of $\PSp_4(2^{f})$ are recorded. We summarise required information in Table~\ref{tbl:conj} below. We also note by \cite[p. 92]{art:Enomoto} that
	\begin{enumerate}[{(i)}]
		\item eight classes $B_t(\pm i,\pm j)$ and $B_t(\pm j,\pm i)$ are the same, for $t=1,4$;
		\item four classes $B_t(\pm i)$, $B_t(\pm qi)$ are the same, for $t=2,5$;
		\item four classes $\B_t(\pm i,\pm j)$ are the same, for $t=3$;
		\item two classes $C_t(\pm i)$ are the same, for $1\leq t \leq 4$;
		\item two classes $D_t(\pm i)$ are the same, for $1\leq t \leq 4$.
	\end{enumerate}
\end{remark}
\begin{table}
	\small
	\caption{Conjugacy classes of $\PSp_4(q)$ with $q=2^{f}$.}\label{tbl:conj}
	\resizebox{\textwidth}{!}{
		\begin{tabular}{@{}llllll@{}}
			\noalign{\smallskip}\hline\noalign{\smallskip}
			Name & Conditions & Representative & Order of Reps. &
			Number of class & Length of class \\ 
			\noalign{\smallskip}\hline\noalign{\smallskip}
			$A_1$ &
			&
			$h(1,1)$ &
			$1$ &
			$1$ &
			$1$  \\
			$A_2$ &
			&
			$x_{2a+b}(1)$ &
			$2$ &
			$1$ &
			$(q^4-1)$ \\
			$A_{31}$ &
			&
			$x_{a+b}(1)$ &
			$2$ &
			$1$ &
			$(q^4-1)$ \\
			$A_{32}$ &
			&
			$x_{a+b}(1)x_{2a+b}(1)$ &
			$2$ &
			$1$ &
			$(q^2-1)(q^4-1)$ \\
			$A_{41}$ &
			&
			$x_a(1)x_b(1)$ &
			$4$ &
			$1$ &
			$\frac{q^2(q^2-1)(q^4-1)}{2}$ \\
			$A_{42}$ &
			$\xi\in \Fbb $ &
			$x_a(1)x_b(1)x_{2a+b}(\xi)$ &
			$4$ &
			$1$ &
			$\frac{q^2(q^2-1)(q^4-1)}{2}$ \\
			$B_1(i,j)$&
			$(i,j)\in S_1$ &
			$h(\gamma^i,\ \gamma^j)$ &
			$\frac{(q-1)}{(q-1,i,j)}$ &
			$\frac{(q-2)(q-4)}{8}$ &
			$q^4(q+1)^{2}(q^2+1)$ \\
			$B_2(i)$ &
			$i\in R_2$ &
			$h(\theta^i,\ \theta^{qi})$ &
			$\frac{(q^{2}-1)}{(q^{2}-1,i)}$&
			$\frac{q(q-2)}{4}$ &
			$q^4(q^4-1)$ \\
			$B_3(i,j)$ &
			$(i,j)\in T_1 \times T_2$ &
			$h(\gamma^i,\ \eta^j)$ &
			$\frac{(q^{2}-1)}{(q-1,i)(q+1,j)}$ &
			$\frac{q(q-2)}{4}$ &
			$q^4(q^4-1)$\\
			$B_4(i,j)$ &
			$(i,j)\in S_2$ &
			$h(\eta^i,\ \eta^j)$ &
			$\frac{q+1}{(q+1,i,j)}$&
			$\frac{q(q-2)}{8}$ &
			$q^4(q-1)^{2}(q^2+1)$ \\
			$B_5(i)$ &
			$i\in R_3$ &
			$h(\tau^i,\ \tau^{qi})$ &
			$\frac{q^{2}+1}{(q^{2}+1,i)}$ &
			$\frac{q^2}{4}$ &
			$q^4(q^2-1)^{2}$ \\
			$C_1(i)$ &
			$i\in T_1$ &
			$h(1,\ \gamma^i)$ &
			$\frac{q-1}{(q-1,i)}$ &
			$\frac{(q-2)}{2}$ &
			$q^3(q+1)(q^2+1)$\\
			$C_2(i)$ &
			$i\in T_1$ &
			$h(\gamma^i,\ \gamma^{-i})$ &
			$\frac{q-1}{(q-1,i)}$ &
			$\frac{(q-2)}{2}$ &
			$q^3(q+1)(q^2+1)$\\
			$C_3(i)$ &
			$i\in T_2$ &
			$h(1,\ \eta^i)$ &
			$\frac{q+1}{(q+1,i)}$ &
			$\frac{q}{2}$ &
			$q^3(q-1)(q^2+1)$\\
			$C_4(i)$ &
			$i\in T_2$ &
			$h(\eta^i,\ \eta^{-i})$ &
			$\frac{q+1}{(q+1,i)}$ &
			$\frac{q}{2}$ &
			$q^3(q-1)(q^2+1)$\\
			$D_1(i)$& $i\in T_1$ &
			$h(1,\ \gamma^i)x_{2a+b}(1)$ &
			$\frac{2(q-1)}{(q-1,i)}$&
			$\frac{(q-2)}{2}$ &
			$q^3(q+1)(q^4-1)$\\
			$D_2(i)$ & $i\in T_1$ &
			$h(\gamma^i,\ \gamma^{-i})x_{a+b}(1)$ &
			$\frac{2(q-1)}{(q-1,i)}$ &
			$\frac{(q-2)}{2}$ &
			$q^3(q+1)(q^4-1)$\\
			$D_3(i)$ &
			$i\in T_2$ &
			$h(1,\ \eta^i)x_{2a+b}(1)$ &
			$\frac{2(q+1)}{(q+1,i)}$ &
			$\frac{q}{2}$ &
			$q^3(q-1)(q^4-1)$\\
			$D_4(i)$ & $i\in T_2$ &
			$h(\eta^i,\ \eta^{-i})x_{a+b}(1)$ &
			$\frac{2(q+1)}{(q+1,i)}$&
			$\frac{q}{2}$ &
			$q^3(q-1)(q^4-1)$\\
			\noalign{\smallskip}\hline\noalign{\smallskip}
		\end{tabular}
	}
\end{table}

\subsection{Elements of the same order}\label{sec:nse}

In this section, we determine the number of the elements of the same order of $\PSp_4(q)$ with $q=2^{f}>2$ in Proposition~\ref{prop:nse(S)} below. 

\begin{lemma}{\rm \cite[Lemma 7]{art:mazurov}}\label{lem:om}
	Let $\S:=\PSp_4(q)$ with $q=2^{f}>2$. If $n$ is an element order of $\S$, then $n$ divides one of the five numbers $4$, $2(q-1)$, $2(q+1)$, $q^2-1$, $q^2+1$.
\end{lemma}
\begin{proof}
	The representative of the conjugacy classes of the group $\S$ is listed as in the third column of Table~\ref{tbl:conj}. Note that each such representative is presented in the form of $x_{\iota}(t)$, $h(z_1,z_2)$ or their products, for some $t\in \Fbb$,  $z_i\in \Fbb_4^\times$, and $\iota\in\Delta^+ :=\{a, b, a+b, 2a+b\}$. The explicit definition of $x_{\iota}(t)$ and $h(z_1,z_2)$ are given in \eqref{eq:x}. It is then easy  to find the order of each representative of the conjugacy classes of the group $\S$ as in the fourth column of Table~\ref{tbl:conj}, and hence the  proof follows.
\end{proof}

\begin{proposition}\label{prop:nse(S)}
	Let $\S:=\PSp_4(q)$ with $q=2^{f}>2$. Then
	\begin{enumerate}[{\rm \quad (i)}]
		\item $\m_2(\S)=(q^2+1)(q^4-1)$;
		\item $\m_4(\S)=q^2(q^2-1)(q^4-1)$;
		\item $\m_{r}(\S)=\varphi(r)q^3(q^2+1)(q+1)\left(1-\frac{q(q+1)}{2}+\frac{q(q+1)}{8}\psi(r)\right)$  if  $1\neq r\mid q-1$;
		\item $\m_{r}(\S)=\varphi(r)q^3(q^2+1)(q-1)\left(1-\frac{q(q-1)}{2}+\frac{q(q-1)}{8}\psi(r)\right)$ if $1\neq r\mid q+1$;
		\item $\m_{2r}(\S)=\varphi(r)q^3(q+1)(q^4-1)$ if $1\neq r\mid q-1$;
		\item $\m_{2r}(\S)=\varphi(r)q^3(q-1)(q^4-1)$ if $1\neq r\mid q+1$;
		\item $\m_{rs}(\S)=\frac{1}{2}\varphi(rs)q^4(q^4-1)$  if $1\neq r\mid q-1$ and $1\neq s\mid q+1$;
		\item $\m_{r}(\S)=\frac{1}{4}\varphi(r)q^4(q^2-1)^2$ if $1\neq r\mid q^2+1$.
		
	\end{enumerate}
\end{proposition}
\begin{proof}
	For $r\in \omega(\S)$, let $\cl_{r}(\S)$ be the set of the conjugacy classes of $\S$ containing all elements of order $r$. Then $\m_r(\S)=\sum_{c\in \cl_r(\S)} |c|$. In what follows, for each $r\in \omega(\S)$, we obtain $\cl_r(\S)$, and then $\m_{r}(\S)$.
	
	If $r=2$, then $\cl_{2}(\S)=\{A_{2}, A_{31}, A_{32}\}$, and so by Table~\ref{tbl:conj}, we have that  $\m_{2}(\S)=|A_{2}|+|A_{31}|+|A_{32}|=(q^{4}-1)+(q^{4}-1)+(q^{2}-1)(q^{4}-1)=
	(q^{2}+1)(q^{4}-1)$. Similarly, $\m_{4}(\S)=|A_{41}|+|A_{42}|=q^2(q^2-1)(q^4-1)/2+q^2(q^2-1)(q^4-1)/2=q^2(q^2-1)(q^4-1)$.
	
	We now prove part (iii). Let $r\neq 1$ be a divisor of $q-1$. By Table~\ref{tbl:conj}, the elements of order $r$ are in the classes $B_1(i,j)$, $C_1(i)$ and $C_{2}(i)$. Thus $\cl_{r}(\S)=B_{1}\cup C_{1}\cup C_{2}$, where
	\begin{align*}
		B_1=& \{B_1(i,j) : (i,j)\in S_1, |h(\gamma^i,\gamma^j)|=r \},\\
		C_1=& \{C_1(i) : i\in T_1, |h(1,\gamma^i)|=r \},\\
		C_2=& \{C_2(i) : i\in T_1, |h(\gamma^i,\gamma^{-i})|=r \}.
	\end{align*}
	We now determine $|B_{1}|$, $|C_{1}|$ and $|C_{2}|$.
	
	Let $\Bmc= \{h(\gamma^i,\gamma^j) : (i,j)\in S_1\}$. Then  We know that 
	\begin{align*}
		\Bmc=& \{h(\gamma^i,\gamma^j) : (i,j)\in S_1\}\\
		=& \{h(\gamma^i,\gamma^j) : 0<i,j<q-1,\ i\not\equiv \pm j\pmod{q-1}\}\\
		=& \{h(z_1,z_2) : z_1,z_2\in \Fbb^\times,\ z_1,z_2\neq 1,\ z_1\neq z_2 \ \text{and}\ z_1\neq z_2^{-1}\}.
	\end{align*}
	Then $\Bmc=\Bmc_{0}\setminus (\Bmc_{1}\cup\Bmc_{2}\cup \Bmc_{3}\cup \Bmc_{4})$, where
	\begin{align*}
		\Bmc_0 &=\{h(z_1,z_2) : z_1,z_2\in \Fbb^\times\}, \\
		\Bmc_1 &=\{h(1,z) : z\in \Fbb^\times\}, \\
		\Bmc_2 &=\{h(z,1) : z\in \Fbb^\times\}, \\
		\Bmc_3 &=\{h(z,z) : z\in \Fbb^\times\}, \\
		\Bmc_4 &=\{h(z,z^{-1}) : z\in \Fbb^\times\}.
	\end{align*}
	Note that $\Bmc_0\cong C_{q-1} \times C_{q-1}$,
	$\Bmc_1\cong  \Bmc_2\cong \Bmc_3\cong \Bmc_4\cong C_{q-1}$,
	and $\Bmc_i \bigcap \Bmc_j=\{h(1,1)\}$, for $1\leq i,j\leq 4$ with $i\neq j$.
	Lemmas \ref{lem:composable} and \ref{lem:G^n} imply that the number of elements of $\Bmc_0$ of order $r$ is equal to $\varphi(r)\psi(r)$. Moreover, by Lemma~\ref{lem:cyclic}, the number of elements of order $r$ in $\Bmc_i$ is equal to $\varphi(r)$, for $i=1,2,3,4$. Therefore, the number $\m_r(\Bmc)$ of elements of order $r$ in $\Bmc$ is equal to $\varphi(r)\psi(r)-4\varphi(r)$. Note by Remark \ref{rem:conj} that $B_1(\pm i,\pm j)=B_1(\pm j,\pm i)$. Then
	\begin{align*}\label{eq:B1}
		|B_{1}|=\frac{1}{8}\m_{r}(\Bmc)=\frac{\varphi(r)}{8} \psi(r)-\frac{\varphi(r)}{2}.
	\end{align*}
	
	We now obtain $|C_{1}|$. Let $s:=(q-1)/r$. Note by Remark~\ref{rem:conj} that $C_1(i)=C_1(-i)$ for all $i\in T_{1}$. Then
	\begin{align*}
		2|C_{1}|&=|\{i\in T_1 : |h(1,\gamma^i)|=r\}|\\
		& =|\{i\in T_1 : \frac{q-1}{(q-1,i)}=r \}|\\
		& =|\{i\in T_1 : 0<i<q-1,\ (q-1,i)=s\}|\\
		& =|\{\frac{i}{s} : 0<\frac{i}{s}<r,\ (r,\frac{i}{s})=1\}|\\
		&=\varphi(r).
	\end{align*}
	Therefore, $|C_{1}|=\varphi(r)/2$. Similarly, $2|C_{2}|=|\{i\in T_1 : |h(\gamma^{i},\gamma^{-i})|=r\}|=\varphi(r)$. Hence
	\begin{align*}
		\m_{r}(\S)=&\sum_{c\in \cl_r(\S)} |c|\\
		=&\sum_{c\in B_{1}} |c|+
		\sum_{c\in C_{1}} |c|+
		\sum_{c\in C_{2}} |c|\\
		=& |B_1(i,j)|\cdot |B_{1}|+|C_1(i')|\cdot |C_{1}|+|C_2(i')|\cdot|C_2|\\
		=&q^4(q+1)^{2}(q^2+1)\cdot \left(-\frac{\varphi(r)}{2}+\frac{\varphi(r)}{8} \psi(r)\right)+\\
		&q^3(q+1)(q^2+1)\cdot \frac{\varphi(r)}{2}+q^3(q+1)(q^2+1)\cdot \frac{\varphi(r)}{2}\\
		=&\varphi(r)q^3(q^2+1)(q+1)\left(1-\frac{q(q+1)}{2}+\frac{q(q+1)}{8}\psi(r)\right).
	\end{align*}
	This proves part (iii). By the same argument as in part (iii), part (iv) follows.
	
	In order to prove part (v), suppose that $r\neq 1$ is a divisor of $q-1$. Then $\cl_{2r}(\S)=D_{1}\cup D_{2}$, where
	\begin{align*}
		D_1=& \{D_1(i) : i \in T_1, |h(1,\gamma^i)x_{2a+b}(1)|=2r \},\\
		D_2=& \{D_2(i) : i\in T_1, |h(\gamma^i,\gamma^{-i})x_{a+b}(1)|=2r \}.
	\end{align*}
	It follows from Remark~\ref{rem:conj} that $D_1(i)=D_1(-i)$ for all $i\in T_{1}$. Therefore, if $s:=(q-1)/r$, then
	\begin{align*}
		2|D_{1}|& =|\{i\in T_1 : |h(1,\gamma^i)x_{2a+b}(1)|=2r\}\\\
		& =|\{i\in T_1 : \frac{2(q-1)}{(q-1,i)}=2r \}|\\
		& =|\{i\in T_1 : \frac{q-1}{(q-1,i)}=r \}|\\
		& =|\{i\in T_1 : 0<i<q-1,\ (q-1,i)=s\}|\\
		& =|\{\frac{i}{s} : 0<\frac{i}{s}<r,\ (r,\frac{i}{s})=1\}|\\
		&=\varphi(r).
	\end{align*}
	Similarly, $|D_{2}|=\varphi(r)/2$. Thus
	\begin{align*}
		\m_{2r}(\S)=&\sum_{c\in \cl_r(\S)} |c|
		=\sum_{c\in D_{1}} |c|+
		\sum_{c\in D_{2}} |c|\\
		=& q^3(q+1)(q^4-1)|D_1|+q^3(q+1)(q^4-1)|D_2|\\ =&q^3(q+1)(q^4-1)\frac{1}{2}\varphi(r)+q^3(q+1)(q^4-1)\frac{1}{2}\varphi(r).
	\end{align*}
	Hence $\m_{2r}(\S)=\varphi(r)q^3(q+1)(q^4-1)$, and this proves part (v). Parts (vi), (vii) and (viii) follows in the same manner as part (v).
\end{proof}

\begin{lemma}\label{lem:phi}
	Let $\S:=\PSp_4(q)$ with $q=2^{f}>2$. If $r>1$ is a divisor of $q^2+1$, then the order of each Sylow $2$-subgroup of $\S$ divides $\m_r(\S)$, and so $4\mid \varphi(r)$.
\end{lemma}
\begin{proof}
	Let $r>1$ be a divisor of $q^2+1$. Then by Lemma~\ref{lem:om}, we observe that $\S$ has no element of order $2r$, and so the Sylow $2$-subgroups of $\S$ act fixed point freely (by conjugation) on the set of elements of order $r$ in $\S$. Thus the order of each Sylow $2$-subgroup of $\S$ divides $\m_r(\S)$. By Proposition~\ref{prop:nse(S)} (viii), $\m_{r}(\S)=\frac{1}{4}\varphi(r)q^4(q^2-1)^2$, and so $q^4\mid  \frac{1}{4}\varphi(r)q^4(q^2-1)^2$. This requires $4\mid \varphi(r)$.
\end{proof}

\section{Proof of the main result}\label{sec:main}
In this section, we prove Theorem \ref{thm:main}. 
For a positive integer $q$, we define the sets $\Amc_{1}, \ldots,\Amc_{9}$ as below:\smallskip

$\Amc_1 = \{1\}$; 

$\Amc_2 = \{(q^2+1)(q^4-1)\}$; 

$\Amc_3 = \{q^2(q^2-1)(q^4-1)\}$; 

$\Amc_4 = \{\varphi(r)q^3(q^2+1)(q+1)\left(1-\frac{q(q+1)}{2}+\frac{q(q+1)}{8}\psi(r)\right): 1\neq r\mid q-1\}$; 

$\Amc_5 = \{\varphi(r)q^3(q^2+1)(q-1)\left(1-\frac{q(q-1)}{2}+\frac{q(q-1)}{8}\psi(r)\right) : 1\neq r\mid  q+1 \}$; 

$\Amc_6 = \{\varphi(r)q^3(q+1)(q^4-1)  :  1\neq r\mid  q-1 \}$; 

$\Amc_7 = \{\varphi(r)q^3(q-1)(q^4-1)  :  1\neq r\mid  q+1 \}$;

$\Amc_8 = \{\frac{1}{2}\varphi(rs)q^4(q^4-1)  :  1\neq r\mid q-1 \text{ and }  1\neq s \mid q+1\}$; 

$\Amc_9 = \{\frac{1}{4}\varphi(r)q^4(q^2-1)^2  :  1\neq r\mid  q^2+1\}$. 

\begin{proposition}\label{prop:m(G)}
	Let $\S:=\PSp_4(q)$ with $q=2^{f}>2$, and let $\Amc_{i}$ with $i=1,\ldots,9$ be as above. Suppose that $G$ is a group such that $\nse(G)=\nse(\S)$ and $|G|=|\S|$. Then
	\begin{enumerate}[{\rm \quad (i)}]
		\item $\m_2(G)\in\Amc_2$;
		\item $\m_{r}(G)\in \Amc_9$ if $r>1$ is a prime divisor of $q^2+1$;
		\item $\m_{r}(G)\in\Amc_4\cup \Amc_5$ if $r>1$ is a prime divisor of $q^2-1$.
	\end{enumerate}
\end{proposition}
\begin{proof}
	Since $\nse(G)=\nse(\S)$, it follows that $\nse(G)$ consists of the positive numbers recorded in parts (i)-(viii) of Proposition~\ref{prop:nse(S)}. Then we observe that $1$ and $(q^2+1)(q^4-1)$ are the only odd numbers of $\nse(G)$, and so Lemma \ref{lem:euler} implies that $\m_2(G)\in \Amc_2$. In order to prove part (ii), suppose that $r>1$ is a prime divisor of $q^2+1$. By Lemma \ref{lem:euler},  $(r,\m_r(G))=1$, and so $\m_r(G)\in \Amc_9$.  Similarly, if $r>1$ is a prime divisor of $q^2-1$, we apply Lemma \ref{lem:euler} and conclude that   $(r,\m_r(G))=1$, and hence $\m_r(G)\in \Amc_4\cup \Amc_5$, as desired in part (iii).
\end{proof}

In what follows, for a positive integer $i$, let $\f(i)$ be the number of elements of order $r$ in $G$ such that $r$ is a multiple of $i$. Let also $\f_j(i)$ be the number of elements of order $r$ in $G$ such that $r$ is a multiple of
$i$ and $\m_r(G)\in\Amc_j$, for $j=1,\cdots,9$, that is to say,
\begin{equation}\label{eq:fp}
	\f(i)=\sum_{i\mid r}\m_r(G)\hspace{1cm}\text{and}\hspace{1cm}
	\f_j(i)=\sum_{\substack{i\mid r \\ \m_r(G)\in\Amc_j}}\m_r(G).
\end{equation}
Therefore, $\f(i)=\sum_{j=1}^{9}\f_j(i)$. It follows from Lemma \ref{lem:euler}  that
$\f(i)=\sum_{j=3}^{9}\f_j(i)$  if $i\geq 3$ \label{eq:fp2}.

\begin{proposition}\label{prop:2component}
	Let $\S:=\PSp_4(q)$ with $q=2^{f}>2$, and let $G$ be a group such that $\nse(G)=\nse(\S)$ and $|G|=|\S|$. Then $\pi(q^2+1)$ and $\pi(2(q^2-1))$ are not in the same connected component of the  prime graph $\Gamma(G)$ of $G$. In particular,  $\Gamma(G)$ has at least two connected components.
\end{proposition}
\begin{proof}
	Suppose, conversely, that $\pi(q^2+1)$ and $\pi(2(q^2-1))$ are in the same connected component of $\Gamma(G)$. Since $\pi(G)$ is disjoint union of $\pi(q^2+1)$ and $\pi(2(q^2-1))$, there exist  $p_{1}\in\pi(q^2+1)$ and $p_{2}\in\pi(2(q^2-1))$ such that $p_{1}p_{2}\in \omega(G)$. So by Lemma \ref{lem:multi}, there exists positive integer $r$ such that $\f(p_{1})=r|G|/|P_{1}|\neq 0$, where $P_{1}\in \Syl_{p_{1}}(G)$. Since the elements of $\Amc_9$ express in the form $\frac{1}{4}\varphi(i)q^4(q^2-1)^2$ where $ 1\neq i\mid  q^2+1$, and by Lemma~\ref{lem:phi}, $4\mid \varphi(i)$, Proposition \ref{prop:m(G)} implies that  there is a positive integer $r'$ such that $\f_9(p_{1})=q^4(q^2-1)^2r'$. Thus
	\begin{align*}
		\f(p_{1})-\f_9(p_{1})=\sum_{k=3}^{8}\f_k(p_{1}).
	\end{align*}
	If $\sum_{k=3}^{8}\f_k(p_{1})\neq 0$, then $q^2+1$ divides $\sum_{k=3}^{8}\f_k(p_{1})=q^4(q^2-1)^2[r(q^2+1)/|P_{1}|-r']$, and so $q^2+1 < r(q^2+1)/|P_{1}|$. This implies that $|G|<\f(p_{1})$, which is impossible. Therefore, $\sum_{k=3}^{8}\f_k(p_{1})= 0$, or equivalently, $\f(p_{1})=\f_{9}(p_{1})$. Thus, $\m_{p_{1}p_{2}}(G)\in\Amc_9$, and hence $\f_9(p_{2})=q^4(q^2-1)^2s\neq 0$, for some positive integer $s$.
	
	Since $p_{2}$ is a prime divisor of $2(q^{2}-1)$ and $q$ is a power of $2$, it follows that $p_{2}=2$ or $p_{2}\in \pi(q^{2}-1)$. We now discuss these two possible cases:\smallskip
	
	\noindent \textbf{(1)} Let $p_{2}=2$. Then Proposition \ref{prop:m(G)} implies that $\m_{p_2}(G)\in\Amc_2$, and so $\f(p_{2})\neq 0$. By Lemma \ref{lem:multi}, there exists a positive integer $t$ such that $\f(p_{2})=(q^4-1)(q^2-1)t$. Thus $\f(p_{2})=\sum_{k=2}^{8}\f_k(p_{2})+\f_{9}(p_{2})$. Since $\sum_{k=2}^{8}\f_k(p_{2})=(q^2+1)u$, for some positive integer $u$, we have that
	\begin{align*}
		(q^2+1)\left((q^2-1)^2t-u\right)=q^4(q^2-1)^2s.
	\end{align*}
	Thus $q^4(q^2-1)^2\mid (q^2-1)^2t-u$. Therefore, $|G|=q^{4}(q^2-1)^2(q^{2}+1)< (q^{2}+1)(q^2-1)^2t\leq \f(p_{2})$, and hence $|G|<\f(p_{2})$, which is a contradiction.\smallskip
	
	\noindent \textbf{(2)} Let $p_{2}\in \pi(q^2-1)$. By Proposition \ref{prop:m(G)}, we have that $\m_{p_{2}}(G)\in\Amc_4 \cup \Amc_5$. Then $\f_4(p_{2})+\f_5(p_{2})\neq 0$. It follows from Lemma \ref{lem:multi} that $\f(p_{2})=t|G|/|P_{2}|\neq 0$, for some positive integer $t$ and Sylow $p_{2}$-subgroup $P_{2}$ of $G$. Thus $\f(p_{2})=\sum_{k=3}^{8}\f_k(p_{2})+\f_{9}(p_{2})$. Note that  $\sum_{k=3}^{8}\f_k(p_{2})=(q^2+1)u$, for some positive integer $u$. Then
	\begin{align*}
		(q^2+1)\left( \frac{q^4(q^2-1)^2}{|P_{2}|}t-u\right)=q^4(q^2-1)^2s,
	\end{align*}
	for some positive integer $s$. Therefore $(q^2+1)\mid s$, and hence $|G|=q^{4}(q^2-1)^2(q^{2}+1)\leq  q^4(q^2-1)^2s=\f_{9}(p_{2})< \f(p_{2})$, that is to say, $|G|<\f(p_{2})$, which is a contradiction.
\end{proof}

\begin{proposition}\label{prop:frob}
	Let $\S:=\PSp_4(q)$  with $q=2^{f}>2$. If $G$ is a group such that $\nse(G)=\nse(\S)$ and $|G|=|\S|$, then $G$ has a normal series $1\unlhd H\unlhd K\unlhd G$ such that
	$H$ and $G/K$ are $\pi_1(G)$-groups, $K/H$ is a non-abelian simple group, $H$ is a nilpotent group, $|G/K|$ divides $|\Out(K/H)|$, and $t(K/H)\geq t(G)$. Moreover, any odd component of $G$ is also an odd component of $K/H$.
\end{proposition}
\begin{proof} By Lemma~\ref{lem:graph} and Proposition~\ref{prop:2component}, it suffices to show $G$ is neither a  Frobenius group, nor a $2$-Frobenius group. Let first $G$ be a Frobenius group with kernel $K$ and complement $H$. Then Lemma \ref{lem:frob} implies that $\pi(H)$ and $\pi(K)$ are the connected components of $\Gamma(G)$. By Proposition \ref{prop:2component}, $\pi(q^2+1)$ and $\pi(2(q^2-1)$ must be the connected components of $\Gamma(G)$, and so $|H|=q^2+1$ and $|K|=q^4(q^2-1)^2$, or $|H|= q^4(q^2-1)^2$ and $|K|=q^2+1$. However, both cases cannot occur as $|H|$ does not divide $|K|-1$.  Let now $G$ be a $2$-Frobenius group. Then by Lemma~\ref{lem:2frob}, there exists normal series $1\unlhd  H\unlhd K\unlhd G$ such that $G/H$ and $K$ are Frobenius groups with kernels $K/H$ and $H$, respectively. Moreover, $\pi_1(G)=\pi(H)\cup \pi(\frac{G}{K})$ and $\pi_2(G)=\pi(\frac{K}{H})$. We now apply Proposition~\ref{prop:2component}, and conclude that $|H|$ divides $q^4(q^2-1)^2$ and $|K/H|=q^2+1$. Since
	$G/H$ is a Frobenius group with kernel $K/H$, it follows that $|G/K|$ divides $|K/H|-1$, and so $|G/K|$ divides $q^2$. Hence $q^2(q^2-1)^2$ is a divisor of $|H|$. This implies that the prime graph of $\Gamma(H)$ is complete because $H$ is nilpotent, and so by Proposition \ref{prop:2component},  $\pi(q^2+1)$ and $\pi(2(q^2-1))$ are connected components of the graph $\Gamma(G)$. Thus $\oc(G)=\oc(\S)$. We  apply Lemma \ref{lem:oc} and conclude  that $G\cong \S$, and hence $\S$ is a $2$-Frobenius group which is a contradiction. This completes the proof.
\end{proof}

\subsection{Proof of Theorem \ref{thm:main}}

Let $\S:=\PSp_4(q)$ be a projective symplectic simple group, where $q=2^{f}$. If $G$ is a group with $\nse(G)=\nse(\S)$ and $|G|=|\S|$, then Proposition~\ref{prop:frob} implies that $G$ has a normal series $1\unlhd H\unlhd K\unlhd G$ such that $H$ and $G/K$ are $\pi_1$-groups, $K/H$ is a non-abelian simple group, $H$ is a nilpotent group, $|G/K|$ divides $|\Out(K/H)|$, and $t(K/H)\geq t(G)$. Moreover, any odd component of $G$ is also an odd component of $K/H$. 

Since $K/H$ is a non-abelian simple group with $t(K/H)\geq t(G)\geq 2$, the group $K/H$ is isomorphic to one of the simple groups listed in \cite{art:kondratev,art:Williams}.
We remark here that since $t(K/H)\geq t(G)\geq 2$, if $K/H$ has two connected components, then $G$ has exactly two connected components, and hence by \cite{art:kondratev,art:Williams}, the odd order components of $K/H$ is equal to $q^{2}+1$.

In what follows, we discuss each of these possibilities for the factor group $K/H$ and the proof of Theorem \ref{thm:main} follows immediately from the following lemmas.

\begin{lemma}\label{lem:AnSp}
	The group $K/H$ cannot be isomorphic to an alternating group, the Tits group $^{2}F_{4}(2)'$ or a sporadic simple group.
\end{lemma}
\begin{proof}
	Let $K/H$ be isomorphic to $\A_{n}$. Then by \cite{art:kondratev,art:Williams}, we have that $n\in\{p,p+1,p+2\}$ and $n\geq 7$. Thus $q^2+1\in\{p,p-2,p(p-2)\}$. If $q^2+1=p$ or $p-2$, then $q^2-3=p-4$ or $q^2+3=p$, and so $q^2\mp 3$ divides $|G|$ and this contradicts Lemma~\ref{lem:q1}(v). If $q^2+1=p(p-2)$, then $(p-1-q) (p-1+q)=2$, which is impossible. If $K/H\cong \A_{n}$, where $n=5,6$, then $q^2+1=3$, $5$ or $15$ which leads to no suitable solution for $q$. Therefore, $K/H$ cannot be isomorphic to an alternating group. Similarly, the case where $K/H$ is the Tits group or a sporadic simple group can be ruled out.
\end{proof}

\begin{lemma}\label{lem:exp}
	The group $K/H$ cannot be isomorphic to a finite simple exceptional group.
\end{lemma}
\begin{proof}
	By \cite{art:kondratev,art:Williams}, the group $K/H$ is isomorphic to one of the simple groups $^{2}\B_{2}(q')$ with $q'=2^{2t+1}\geq 8$, ${}^3\D_4(q')$, $\G_2(q')$, ${}^2\G_2(q')$ with $q'=3^{2t+1}\geq 27$, $\F_4(q')$, ${}^2\F_4(q')$ with $q'=2^{2t+1}\geq 8$, $\E_6^{\e}(q')$ with $\e=\pm$, $\E_6^{-}(2)$, $\E_7(2)$, $\E_7(3)$ and $\E_8(q')$.
	
	We easily see that $K/H$ cannot be isomorphic to $\E_6^{-}(2)$, $\E_7(2)$ or $\E_7(3)$ as if $q^{2}+1$ has to be equal to one of the following numbers which gives no possible solution.
	\begin{align*}
		\E_6^{-}(2):&\quad  13,17,19,13\cdot17,13\cdot19,17\cdot19,13\cdot17\cdot 19,\\
		\E_7(2): &\quad 73,127,73\cdot127,\\
		\E_7(3): &\quad 757,1093,757\cdot1093.
	\end{align*}
	
	Suppose that $K/H$ is isomorphic to one of the groups  listed in the first column of Table~\ref{tbl:exp}. Then $q^{2}+1$ is equal to one of the polynomials recorded in the third column of the  same table. But such an equality contradicts the fact that $q$ is a power of $2$.
	
	Suppose that $K/H$ is isomorphic to a Suzuki group $^{2}\B_{2}(q')$. Then $q^2+1\in\{q'-1,q'\pm \sqrt{2q'}+1,q'^2+1,(q'-1)(q'\pm \sqrt{2q'}+1), (q'-1)(q'^2+1) \}$. If  $q^2+1\neq q'^2+1$, then we conclude that $q$ is not a power of $2$, which is a  contradiction. Thus $q^2+1=q'^2+1$, and so $q=q'$. Clearly, every two vertices $r$ and $s$ of $\Gamma(K/H)$ with $r\in\pi(q'-\sqrt{2q'}+1)$ and $s\in\pi(q'+\sqrt{2q'}+1)$ are disconnect, and hence they are also disconnected in $\Gamma(G)$. Therefore the Sylow $s$-subgroup of $G$ acts fixed point freely on $\m_r(G)$. This implies that $q'+\sqrt{2q'}+1\mid \m_r(G)$. On the other hand, it follows from Lemma \ref{prop:m(G)} that $\m_r(G)\in \A_9$. Therefore, $q'+\sqrt{2q'}+1$ divides $\varphi(r)q^4(q^2-1)^2/4=\varphi(r)q'^4(q'^2-1)^2/4$, and hence  $q'+\sqrt{2q'}+1$ divides $\varphi(r)$, where $r\in\pi(q'-\sqrt{2q'}+1)$, which is a contradiction.

	Suppose that $K/H$ is isomorphic to $F_4(q')$. Then $q^2+1\in\{q'^4+1,q'^4-q'^2+1, q'^8-q'^6+2q'^4-q'^2+1 \}$. If $q^2+1=q'^4-q'^2+1$ or $q'^8-q'^6+2q'^4-q'^2+1$, then $q^2= q'^4-q'^2$ or $q'^8-q'^6+2q'^4-q'^2$, which is a contradiction as $q$ is a power $2$. Thus $q^2+1=q'^4+1$, and so $q^{12}=q'^{24}$. Since $q'^{24}$ divides $|K/H|$, so does $|G|$, and hence the $2$-part of $|G|$ is at least $q^{12}$, which is a contradiction.
	
	Suppose that $K/H$ is isomorphic to $\E_6^{\e}(q')$ with $q'\equiv \e1 \pmod{3}$ for $\e=\pm$. Then $q^2+1=(q'^6+\e q'^3+1)/3$, and so $3q^2+2=q'^6+\e q'^3$. Note that $q'^6+\e q'^3$ divides $|K/H|$, so does $|G|$. Thus $3q^2+2$ divides $|G|$ which is impossible by Lemma \ref{lem:q1}.
	%
\end{proof}

\begin{table}
	\small 
	\caption{Some polynomials in {\rm Lemma~\ref{lem:exp}}.}\label{tbl:exp}
	\resizebox{\textwidth}{!}{
		\begin{tabular}{@{}lll@{}}%
			\noalign{\smallskip}\hline\noalign{\smallskip}
			Group & Conditions & Polynomials  \\ 
			\noalign{\smallskip}\hline\noalign{\smallskip}
			$\G_2(q')$ &  & $\Phi_{3}(q')$, $\Phi_{6}(q')$, $\Phi_{3}(q'^2)$\\
			$\E_6(q') =\E_6^{+}(q')$ & $\mmod{q'}{0,-1}{3}$ &  $\Phi_{9}(q')$ \\
			$\E_8(q')$ 
			&  & $\Phi_{15}(q')$, $\Phi_{20}(q')$, $\Phi_{24}(q')$, $\Phi_{30}(q')$, \\
			&  &  $\Phi_{15}(q')\Phi_{20}(q')$, $\Phi_{15}(q')\Phi_{24}(q')$, $\Phi_{15}(q') \Phi_{30}(q')$,\\
			&  &  $\Phi_{20}(q')\Phi_{24}(q')$, $\Phi_{20}(q')\Phi_{30}(q')$,\\ 
			&  &  $\Phi_{24}(q')\Phi_{30}(q')$,
			$\Phi_{15}(q')\Phi_{20}(q')\Phi_{24}(q')$,\\
			&  & $\Phi_{15}(q')\Phi_{20}(q')\Phi_{30}(q')$,
			$\Phi_{15}(q')\Phi_{24}(q')\Phi_{30}(q')$,\\
			&  & $\Phi_{20}(q')\Phi_{24}(q')\Phi_{30}(q')$, $\Phi_{15}(q')\Phi_{20}(q')\Phi_{24}(q')\Phi_{30}(q')$\\
			${}^3\D_4(q')$ &  & $\Phi_{12}(q')$\\
			${}^2\G_2(q')$  &  & $\Phi_{6}^{+}(q')$, $\Phi_{6}^{-}(q')$, $\Phi_{6}(q')$ \\
			${}^2\F_4(q')$ &   & $\Phi_{12}^{+}(q')$,$\Phi_{12}^{-}(q')$,$\Phi_{12}(q')$\\
			$^{2}\E_6(q')=\E_6^{-}(q')$ & $\mmod{q'}{0,1}{3}$ &  $\Phi_{18}(q')$ \\	
			\noalign{\smallskip}\hline\noalign{\smallskip}
		\end{tabular}
	}
	
	\footnotetext{$\Phi_{n}(q')$ is the $n$th cyclotomic polynomial.}
	\footnotetext{$\Phi_{6}(q')=\Phi_{6}^{+}(q')\Phi_{6}^{-}(q')$, where $\Phi_{6}^{\e}(q')=q'+\e\sqrt{3q'}+1$ for $\e=\pm$.}
	\footnotetext{$\Phi_{12}(q')=\Phi_{12}^{+}(q')\Phi_{12}^{-}(q')$, where  $\Phi_{12}^{\e}(q')=q'^{2}+\e q'\sqrt{2q'}+q'+\e\sqrt{2q'} +1$ for $\e=\pm$.}
\end{table}

\begin{lemma}\label{lem:psl}
	If $K/H$ is isomorphic to $\PSL_{n}(q')$ with $(n,q')\neq (2,2),(2,3)$, then $(n,q')=(2,q^2)$ with $q=2^f$, and hence $G\cong \S$.   	
\end{lemma}

\begin{proof}
	By \cite{art:kondratev,art:Williams}, we have to consider the following possibilities:\smallskip
	
	\noindent \textbf{(1)} $n\geq 5$ is prime. It follows from   \cite{art:kondratev,art:Williams} that $q^2+1=(q'^n-1)/[(q'-1)(n,q'-1)]$.
	Suppose that $q'$ is even. If $q'=2$, then $q^2+1=q'^n-1$, and so  $2^{2f-1}+1=2^{n-1}$, which is a contradiction. Thus $q'>2$, and hence $q'^{n}-1=(q'-1)(n,q'-1)(q^2+1)>q^2+1$. Then $q'^{2n}>q^4$. Note that the $2$-part of $|K/H|$ is $q'^{\frac{n(n-1)}{2}}$ and  the  $2$-part of $G$ is $q^4$. Since  $|K/H|$ divides $|G|$, we have that $n(n-1)/2<2n$ which is valid only for $n<5$, which is a contradiction. If $q'$ is odd, then $2q^2<q'^n-1$, and so  $q'^{\frac{n(n-1)}{2}}$ divides $(q-1)^2$ or $(q+1)^2$ as $|K/H|$ is a divisor of $G$. Thus $2q^2< q'^n-1 \leq q'^{\frac{n(n-1)}{2}} \leq (q+ 1)^2$, and so  $2q^2<(q+1)^2$. Therefore, $q< 3$, which is a contradiction.

	\noindent \textbf{(2)} $n=p+1$, where $p$ is an odd prime and $q'-1\mid p+\e1$ with $\e=\pm$. In this case, $q^2+1=(q'^p-1)/(q'-1)$, that is to say,  $q^2=q'^{p-1}+q'^{p-2}+\cdots +q'=q'(q'^{p-1}-1)/(q'-1)$. Thus $p=2$ and $q'=q^2$ which is a contradiction.
	
	\noindent \textbf{(3)} $n=3$ and $q'$ is even. If $q'=2,4$, then by \cite{art:kondratev,art:Williams}, $q^2+1$ is equal to one of $3$, $5$, $7$, $15$, $21$, $35$ and $105$ which leads to no possible solution for $q$. Therefore,  $q'\neq 2,4$,  and hence $q^2+1=(q'^2+q'+1)/(3,q'-1)$. Thus $q'(q'+1)=q^2$ or $3q^2+2$. Since both $q$  and $q'$ are even, we deduce that $q'(q'+1)=3q^2+2$. Moreover, $3q^2+2$ divides $|K/H|$, and so does $|G|$, and this contradicts Lemma~\ref{lem:q1}.
	
	\noindent \textbf{(4)} $n=2$ and $q'$ is odd and $q' \equiv \pm 1\pmod 4$. In this case, we must have $q^2+1=q'$, $(q'\mp 1)/2$ or $q'(q'\mp 1)/2$. If $q^2+1=q'$, then $q^2+2=q'+1$ which is a divisor of $|K/H|$. This implies that $q^2+2$ divides $|G|$, and so by Lemma~\ref{lem:q1}(ii) and the fact that $q>2$, we conclude that $q=4$, and hence $q'=q^{2}+1=17$. Therefore, $K/H\cong \PSL_{2}(17)$ and $\S=\PSp_{4}(4)$. Since $|G/K|$ divides $|\Out(K/H)|=2$, we deduce that $|H|=2^{a}\cdot5^{2}$, where $a=3,4$. Since $H$ is nilpotent, it follows that the Sylow $5$-subgroup of $H$ is normal in $G$, and so $\m_5(H)=\m_5(G)$. Therefore $m_5(G)=m_5(H)$ divides $|H|$, but $|H|\mid 400$, and this contradicts Proposition \ref{prop:m(G)}.
	
	If $q^2+1=\frac{q'-\e1}{2}$ with $\e=\pm$, then $q'=2q^2+3$ or $2q^2+1$ respectively when $\e=-$ or $+$. Thus $2q^2+3$ or $2q^2+1$ is a divisor of $|G|$, which is impossible  Lemma~\ref{lem:q1} as $q\neq 2$. Similarly, if $q^2+1=\frac{q'(q'-\e 1)}{2}$ with $\e=\pm$,  then $2q^2=(q'+\e 1)(q'-\e 2)$ which is a contradiction because $q$ is even and $q'\mp 2$ is odd.
	
	\noindent \textbf{(5)} $n=2$ and $q'$ is even. Then $q^{2}+1$ is equal to $q'-1,q'+1,q'^2-1$. We first observe that $q^2+1$ cannot be equal to $q'-1$ or $q'^2-1$ as if not $q^2+2=q'$ or $q'^2$, respectively, which is impossible. Thus  $q^2+1=q'+1$, and hence $q^2=q'$. Note that $\pi(q^2+1)$  is a connected component of  $\Gamma(K/H)$. Then Proposition~\ref{prop:2component}, $\pi(q^2+1)$ is a connected component of $\Gamma(G)$. We show that $\pi(2(q^2-1))$ is the other connected component of $\Gamma(G)$. If $|H|$ is coprime to $q^2-1$, then $q^2-1$ divides $|G/K|$ which is a divisor of  $|\Out(K/H)|=\log_{2}q'=2\log_{2}q$, and so $q^2-1\mid 2\log_{2}q$, which is impossible. Thus there is a prime $p\in \pi(q^2-1)$ such that $p\mid |H|$. On the other hand $2\in \pi(H)$. Since $H$ is a nilpotent group, we have $2p\in\omega(H)$, and this implies that $2p\in\omega(G)$. Therefore, $\pi(2(q^2-1))$ is the other connected component of $\Gamma(G)$. This together with that fact that $|G|=|S|$ implies that  $\oc(G)=\oc(S)$, and hence by Lemma~\ref{lem:oc}, we conclude that $G\cong \S$.
\end{proof}

\begin{lemma}\label{lem:Psu}
	The group $K/H$ cannot be isomorphic to $\PSU_{n}(q')$.
\end{lemma}
\begin{proof}
	According to \cite{art:kondratev,art:Williams}, we have the following possibilities:\smallskip 
	
	\noindent \textbf{(1)} $n=4,6$ and $q'=2$. Then	$q^2+1$ is equal to one of $5$, $7$, $11$, and $77$ which leads to no possible solution for $q$.	
	
	\noindent \textbf{(2)} $n=p+1$, or $n=p$ and $(q'+1,p)=1$ . Then $q^2+1=\frac{q'^p+1}{q'+1}$ and so $q^2=q'^{p-1}-q'^{p-2}+\cdots -q'$, which is a contradiction as $q$ is a power $2$.
	
	\noindent \textbf{(3)} $n=p$ and $(q'+1,p)=p$. Then $q^2+1=\frac{q'^p+1}{(q'+1)\cdot p}$. If $q'$ is even, 
	then the $2$-part $|K/H|_{2}=q'^{\frac{p(p-1)}{2}}$ of $K/H$ divides $|G|_{2}=q^4$. It follows from $q^2+1=\frac{q'^p+1}{(q'+1)\cdot p}$ that $q^{4}<q'^{2p}$, and so $q'^{\frac{p(p-1)}{2}}\leq q^{4}<q'^{2p}$. Then $\frac{p(p-1)}{2}<2p$. This inequality is true only for $p=3$ in which case we have that $3q^2=(q'+1)(q'-2)$, which is impossible. Suppose now that $q'$ is odd. Since $|K/H|$ is a divisor of $|G|$, it follows that  $q'^{\frac{p(p-1)}{2}}\mid (q \pm 1)^2$.  On the other hand, $q^2+1=\frac{q'^p+1}{(q'+1)\cdot p}$ implies that  $2q^2< q'^p+1$. If $p>3$, then $2q^2< q'^p+1\leq q'^{\frac{p(p-1)}{2}} < (q\pm 1)^2$, that is to say, $2q^{2}<(q\pm 1)^2$ which implies that  $q<3$, which is a contradiction. If $p=3$, then  $q^2+1=\frac{q'^3+1}{3(q'+1)}$, and so  $3q^2+2$ divides $|G|$, and so Lemma~\ref{lem:q1}(iv) implies that $q=4$. Thus $\frac{q'^3+1}{3(q'+1)}=q^{2}+1=17$, and so $q'^{2}-q'-50=0$, which has no possible solution. 
\end{proof}

\begin{lemma}\label{lem:Psp}
	If $K/H$ is isomorphic to $\PSp_{2n}(q')$, then $(n,q')=(2,q)$ with $q=2^f$, and hence $G\cong \S$.
\end{lemma}
\begin{proof}
	By \cite{art:kondratev,art:Williams}, we have the following possibilities:\smallskip\\
	\noindent \textbf{(1)} $n=p$ is prime,  and $q'=2,3$. Then $q^2+1=(q'^p-1)/(2,q'-1)$, and so $2q^2=3^p-3$ or $q^2+2=2^p$, which is a contradiction because $q$ is a power $2$.
	
	\noindent \textbf{(2)} $n={2^m}\geq 2$. Then $q^2+1=(q'^n+1)/(2,q'-1)$. If $q'$ is odd, then $2q^2+1=q'^n$, and so Lemma~\ref{lem:numb} implies that $n=1$, which is a contradiction. Therefore, $q'$ is even, and hence $q^2+1=q'^n+1$ which implies that  $q^{2n}=q'^{n^2}$. Since $|K/H|$ divides $|G|$, we have that $q'^{n^2}\mid q^4$, and so  $q^{2n}\mid q^4$. Thus $n=2$, and hence $q=q'$. Therefore $K/H\cong \PSp_4(q)$, and since $|K/H|=|G|$, we conclude that $H=1$, that is to say, $G=K\cong \PSp_4(q)$.
\end{proof}

\begin{lemma}\label{lem:Pom}
	The group $K/H$ cannot be isomorphic to $\POm_{n}^{\e}(q')$ with $\e\in\{\circ,+,-\}$.
\end{lemma}
\begin{proof}
	According to \cite{art:kondratev,art:Williams}, we have the following possibilities:\smallskip
	
	\noindent \textbf{(1)} $n=2m+1$ and $\e=\circ$ with $m={2^t}$ for  $t\geq 2$. Then  $q^2+1=\frac{q'^m+1}{2}$,  and so  $2q^2+1=q'^m$, which is impossible by Lemma~\ref{lem:numb}.
	
	\noindent \textbf{(2)} $n=2m+1$ and $\e=\circ$ with $m$ a prime number and $q'=3$. Then $q^2+1=\frac{3^m-1}{2}$ and so  $2q^2=3^m-3$ which is a contradiction because $q$ is a power $2$.
	
	\noindent\textbf{(3)} $n=2m$, $\e=+$, $m\geq 5$ is odd prime and $q'=2,3,5$. Then $q^2+1=(q'^m-1)/(q'-1)$, Thus $q^2=q'^{m-1}+q'^{m-2}+\ldots+1$ which is a contradiction.
	
	\noindent\textbf{(4)} $n=2(m+1)$, $\e=+$, $m$ is odd prime and $q'=3$. Then $q^2+1=(3^m-1)/2$, Thus $2q^2=3^m-3$, which is a contradiction.

	\noindent \textbf{(5)} $n=2m$, $\e=-$ and $m=2^t\geq 4$. If $q'$ is odd, then $q^2+1=\frac{q'^m+1}{2}$  and so $2q^2+1=q'^m$. Thus $2q^2+1$ divides $|K/H|$, and hence  $2q^2+1$ divides $|G|$, which is a contradiction. If $q'$ is even, then $q^2+1=q'^m+1$, and so $q^{2(m-1)}=q'^{m(m-1)}$. Since $|K/H|$ divides $|G|$, we deduce that $q^{2(m-1)} \mid q^4$. This implies that $m-1 \mid 2$, and so $m=2$ or $3$ which contradicts the fact that $m=2^t\geq 4$.
	
	\noindent \textbf{(6)} $n=2m$, $\e=-$, $2^t+1\neq m \geq  5$ is odd prime and $q'=3$. Then $q^2+1=\frac{3^m+1}{4}$,  and so $4q^2=3^m-3$,  which is a contradiction as $q$ is a power $2$.
	
	\noindent \textbf{(7)} $n=2m$, $\e=-$, $5\leq m=2^t+1$ is not prime and $q'=3$. Then $q^2+1=\frac{3^{m-1}+1}{2}$  and so $2q^2+1=3^{m-1}$. By Lemma~\ref{lem:numb}, $m=3$, which is impossible as $m\geq 5$.
	
	\noindent \textbf{(8)} $n=2m$, $\e=-$, $5\leq m=2^t+1$ is an odd prime and $q'=3$. Then $q^2+1=\frac{3^{m-1}+1}{2}$, $\frac{3^{m}+1}{4}$ or $(\frac{3^{m-1}+1}{2})(\frac{3^{m}+1}{4})$. If  $q^2+1=\frac{3^{m-1}+1}{2}$, then $2q^2+1=3^{m-1}$, and since $3^{m-1}\mid |K/H|$, it follows that $2q^2+1\mid |G|$, and so by Lemma \ref{lem:q1}, we conclude that $q=2$, which is a contradiction. If  $q^2+1=\frac{3^{m}+1}{4}$, then $4q^2=3^{m}-3$,  which is a contradiction because $q$ is a power $2$. If $q^2+1=(\frac{3^{m-1}+1}{2})(\frac{3^{m}+1}{4})$. Since $3^{m(m-1)}$ divides $|K/H|$, it follows that $3^{m(m-1)} \mid |G|$, and so   $3^{m(m-1)} \mid (q\pm 1)^2$. Then  $3^{m(m-1)/2}\leq q+1$, and so $3^{m(m-1)/2} \leq q+1 < q^2+1 \leq (\frac{3^{m-1}+1}{2})(\frac{3^{m}+1}{4}) \leq 3^{2m}$, that is to say, $m(m-1)/2 \leq 2m$. This is true only if $m=5$ in which case $q^2+1=(\frac{3^{4}+1}{2})(\frac{3^{5}+1}{4})=2501$, and so $q=50$,  which is a contradiction.
	
	\noindent \textbf{(9)} $n=2(m+1)$, $\e=-$, $m\neq 2^t-1$ prime and $q'=2$. Then $q^2+1=2^m-1$, which is a contradiction as $q$ is a power of $2$.
	
	\noindent \textbf{(10)} $n=2m$, $\e=-$, $m=p+1$ with $p$ odd prime, and $q'=2$. Then $q^2+1=2^p+1$, $2^{p+1}+1$ or $(2^p+1)(2^{p+1}+1)$. If $q^2+1=2^p+1$, then $p=2f$, which is a contradiction because $p$ is odd prime. If $q^2+1=2^{p+1}+1$, then $2^{4f}=2^{2p+2}$. Since $|K/H| $ divides $|G|$ and the $2$-part of $|G|$ is $q^{4}=2^{4f}$, we deduce that $2^{p(p+1)} \mid 2^{4f}$. This implies that $p\leq 2$, which is a contradiction. If $q^2+1=(2^p+1)(2^{p+1}+1)$, then $q^2=2^{2p+1}+2^p+2^{p+1}$, which is a contradiction because $q$ is a power $2$.
\end{proof}

\begin{proof}[\bf Proof of Theorem \ref{thm:main}]
	The proof follows immediately from Proposition~\ref{prop:frob} and Lemmas \ref{lem:AnSp}-\ref{lem:Pom}.	
\end{proof}

\begin{proof}[\bf Proof of Corollary~\ref{cor:main}]
	Since $G$ and $S$ are of the  same type, it follows that $\tau(G)=\tau(S)$ and $|G|=|S|$, and hence the result immediately follows  from Theorem \ref{thm:main}.	
\end{proof}

\section*{Declaration of competing interest}

The authors confirm that this manuscript has not been published elsewhere. It is not also under consideration by another journal. They also confirm that there are no known conflicts of interest associated with this publication. The authors have no competing interests to declare that are relevant to the content of this
article and they confirm that availability of data and material is not applicable.  The list of authors is in alphabetic order and all authors contributed equally to this manuscript.

\section*{Acknowledgments}

The authors are grateful to the anonymous referees for helpful and constructive comments. The authors would like to thank Andrey Vasil$'$ev for his valuable comments on Thompson's problem.


\begin{thebibliography}{10}
	
	\bibitem{art:Alavi2017}
	S.~H. Alavi, A.~Daneshkhah, and H.~Parvizi~Mosaed.
	\newblock On quantitative structure of small {R}ee groups.
	\newblock {\em Comm. Algebra}, 45(9):4099--4108, 2017.
	
	\bibitem{art:Alavi2019}
	S.~H. Alavi, A.~Daneshkhah, and H.~Parvizi~Mosaed.
	\newblock Finite groups of the same type as {S}uzuki groups.
	\newblock {\em Int. J. Group Theory}, 8(1):35--42, 2019.
	
	\bibitem{book:Car}
	R.~W. Carter.
	\newblock {\em Simple groups of {L}ie type}.
	\newblock John Wiley \& Sons, London-New York-Sydney, 1972.
	\newblock Pure and Applied Mathematics, Vol. 28.
	
	\bibitem{art:Chen}
	G.~Chen.
	\newblock On structure of {F}robenius and $2$-{F}robenius group.
	\newblock {\em Jornal of Southwest China Normal University}, 20(5):485--487,
	1995.
	
	\bibitem{book:atlas}
	J.~H. Conway, R.~T. Curtis, S.~P. Norton, R.~A. Parker, and R.~A. Wilson.
	\newblock {\em Atlas of finite groups}.
	\newblock Oxford University Press, Eynsham, 1985.
	\newblock Maximal subgroups and ordinary characters for simple groups, With
	computational assistance from J. G. Thackray.
	
	\bibitem{art:Crescenzo}
	P.~Crescenzo.
	\newblock A {D}iophantine equation which arises in the theory of finite groups.
	\newblock {\em Advances in Math.}, 17(1):25--29, 1975.
	
	\bibitem{art:Enomoto}
	H.~Enomoto.
	\newblock The characters of the finite symplectic group {${\rm Sp}(4,\,q)$},
	{$q=2\sp{f}$}.
	\newblock {\em Osaka Math. J.}, 9:75--94, 1972.
	
	\bibitem{book:Gor}
	D.~Gorenstein.
	\newblock {\em Finite groups}.
	\newblock Chelsea Publishing Company., New York, second edition, 1980.
	
	\bibitem{book:Hall}
	M.~Hall.
	\newblock {\em The theory of groups}.
	\newblock The Macmillan Co., New York, N.Y., 1959.
	
	\bibitem{book:khukh}
	E.~I. {Khukhro} and V.~D. {Mazurov}.
	\newblock {Unsolved Problems in Group Theory. The Kourovka Notebook}.
	\newblock {\em arXiv e-prints}, page arXiv:1401.0300, Jan. 2014.
	
	\bibitem{art:kondratev}
	A.~S. Kondrat{\cprime}ev.
	\newblock On prime graph components of finite simple groups.
	\newblock {\em Mat. Sb.}, 180(6):787--797, 864, 1989.
	
	\bibitem{art:mazurov}
	V.~D. Mazurov, M.~C. Su, and C.~P. Chao.
	\newblock Recognition of the finite simple groups {$L_3(2^m)$} and {$U_3(2^m)$}
	from the orders of their elements.
	\newblock {\em Algebra Log.}, 39(5):567--585, 631, 2000.
	
	\bibitem{book:rose}
	J.~S. Rose.
	\newblock {\em A course on group theory}.
	\newblock Courier Corporation, 1994.
	
	\bibitem{art:Shao}
	C.~Shao, W.~Shi, and Q.~Jiang.
	\newblock Characterization of simple {$K_4$}-groups.
	\newblock {\em Front. Math. China}, 3(3):355--370, 2008.
	
	\bibitem{a:spec}
	A.~V. Vasil\cprime~ev, M.~A. Grechkoseeva, and V.~D. Mazurov.
	\newblock Characterization of finite simple groups by spectrum and order.
	\newblock {\em Algebra Logika}, 48(6):685--728, 821, 824, 2009.
	
	\bibitem{art:Weisner}
	L.~Weisner.
	\newblock On the number of elements of a group which have a power in a given
	conjugate set.
	\newblock {\em Bull. Amer. Math. Soc.}, 31(9-10):492--496, 1925.
	
	\bibitem{art:Williams}
	J.~S. Williams.
	\newblock Prime graph components of finite groups.
	\newblock {\em J. Algebra}, 69(2):487--513, 1981.
	
	\bibitem{art:zhang}
	L.~Zhang and W.~Shi.
	\newblock A new characterization of the group {$S_4(q)$} using a graph of
	noncommutativity.
	\newblock {\em Sibirsk. Mat. Zh.}, 50(3):669--679, 2009.
	
\end{thebibliography}

\def\cprime{$'$}

\end{document}